\documentclass[12pt]{article}
\usepackage[english]{babel}
\usepackage{amsmath,amsfonts,amssymb,amsthm}
\usepackage{enumerate}
\usepackage{tikz}
\usepackage{url}
\usepackage[left=3cm, right=3cm, top=3cm, bottom=3cm]{geometry}

\newcommand{\norm}[1]{{\left\| #1 \right\|}}
\newcommand{\abs}[1]{{\left|{#1}\right|}}

\newcommand{\<}{\langle}
\renewcommand{\>}{\rangle}

\DeclareMathOperator{\volume}{vol}
\DeclareMathOperator{\mspan}{span}

\newcommand{\N}{\mathbb{N}}
\newcommand{\Z}{\mathbb{Z}}
\newcommand{\Q}{\mathbb{Q}}
\newcommand{\R}{\mathbb{R}}

\newcommand{\F}{\mathbb{F}}

\newcommand{\calB} {\mathcal{B}}

\newcommand{\calP} {\mathcal{P}}
\newcommand{\calQ} {\mathcal{Q}}

\theoremstyle{plain}
\newtheorem{definition}{Definition}[section]
\newtheorem{theorem}[definition]{Theorem}
\newtheorem{proposition}[definition]{Proposition}
\newtheorem{lemma}[definition]{Lemma}
\newtheorem{corollary}[definition]{Corollary}
\newtheorem{conjecture}[definition]{Conjecture}
\theoremstyle{definition}
\newtheorem{remark}[definition]{Remark}

\title{On the Probability of Generating a Lattice}
\author{Felix Fontein\footnote{Insitute of Mathematics, University of Zurich, Winterthurerstrasse 190,
    8057 Zurich, Switzerland; \texttt{felix.fontein@math.uzh.ch}}
  \and Pawel Wocjan\footnote{Mathematics Department \& Center for Theoretical Physics, Massachusetts
    Institute of Technology, Cambridge, MA 02139, USA; on sabbatical leave from Department of
    Electrical Engineering and Computer Science, University of Central Florida, Orlando, FL 32816,
    USA; \texttt{wocjan@eecs.ucf.edu}}}

\begin{document}
  \maketitle
  
  \begin{abstract}
    We study the problem of determining the probability that $m$ vectors selected uniformly at
    random from the intersection of the full-rank lattice $\Lambda$ in $\R^n$ and the window
    $[0,B)^n$ generate $\Lambda$ when $B$ is chosen to be appropriately large.  This problem plays
    an important role in the analysis of the success probability of quantum algorithms for solving
    the Discrete Logarithm Problem in infrastructures obtained from number fields and also for
    computing fundamental units of number fields.
    
    We provide the first complete and rigorous proof that $2n+1$ vectors suffice to generate
    $\Lambda$ with constant probability (provided that $B$ is chosen to be sufficiently large in
    terms of $n$ and the covering radius of $\Lambda$ and the last $n+1$ vectors are sampled from a
    slightly larger window).  Based on extensive computer simulations, we conjecture that only $n+1$
    vectors sampled from one window suffice to generate $\Lambda$ with constant success probability.
    If this conjecture is true, then a significantly better success probability of the above quantum
    algorithms can be guaranteed.
  \end{abstract}
  
  \section{Introduction}
  
  Let $G$ be a finite group.  Denote by $p_m(G)$ the probability that $m$ elements drawn uniformly
  at random from $G$ with replacement generate $G$.  The problem of determining or bounding this
  probability is of fundamental interest in group theory and has been extensively studied for
  various families of groups \cite{acciaro,pomerance-generate}.
  
  The purpose of this paper is to study a very natural generalization of this problem from finite
  abelian groups to finitely generated abelian torsion-free groups.  More precisely, we consider the
  case of lattices, i.e., discrete subgroups of $\R^n$.  The problem is now to determine the
  probability that $m$ vectors selected uniformly at random with replacement from the intersection
  of the full-rank lattice $\Lambda$ in $\R^n$ and a window $X\subset\R^n$ generate $\Lambda$.  We
  denote this probability by $p_m(\Lambda,X)$.
  
  Our study this problem of was initially motivated by its relevance to quantum algorithms and
  quantum cryptanalysis, which we explain in more detail at the end of the paper.  But we also
  believe that this problem is interesting on its own due to its appeal as a very natural and
  fundamental problem in lattice theory.  In fact, it can be viewed as a generalization of the
  following elementary problem in number theory.  For $\Lambda=\Z$ and $X=[1,B]$, the probability
  $p_m(\Lambda,X)$ corresponds to the probability that $m$ integers chosen uniformly at random from
  the set $\{1,\ldots,B\}$ with replacement are coprime.  It is known that
  $\lim_{B\rightarrow\infty} p_m(\Z,B)=1/\zeta(m)$ where $\zeta$ denotes the Riemann zeta function.
  For $\Lambda=\Z^n$ and $X$, the probability $p_m(\Lambda,X)$ is equal to the probability that the
  $m\times n$ matrix whose column vectors are selected uniformly at random from $\Lambda\cap X$ is
  unimodular.  This problem was studied for special forms of $X$ asymptotically. For $X = [-B,
  B]^n$, $B \to \infty$, it was studied by G.~Maze, J.~Rosenthal and U.~Wagner in
  \cite{naturaldensityunimodular}, and for $X = v + [-B, B]^n$, $B \to \infty$, where the entries of
  the vector $v$ are bounded polynomially in terms of $B$, by S.~Elizalde and K.~Woods in
  \cite{elizalde-woods}.  In both cases, it was shown that the limit of the probability for $B \to
  \infty$ is $\prod_{j=m-n+1}^m \zeta(j)^{-1}$.  Both works did not study the problem of bounding
  the probability in the non-asymptotic case, i.e., in the case where $B$ is fixed.
  
  In this paper, we consider the case where $\Lambda$ is an arbitrary full-rank lattice and
  $X=[0,B)^n$ for a sufficiently large but fixed $B$.  Ideally, we want to minimize $m$, while at
  the same time ensure that the probability $p_m(\Lambda,[0,B))^n$ is bounded from below by a
  nonzero constant.  We use $\nu(\Lambda)$ to denote the covering radius of $\Lambda$,
  $\lambda_1(\Lambda)$ the length of a shortest (nonzero) vector of $\Lambda$, and $\det (\Lambda)$
  the determinant of $\Lambda$.
  
  Our two major contributions to the study of this problem are: 
  \begin{theorem}
    \label{thm:main}
    Let $\Lambda$ be a lattice of full rank in $\R^n$, and assume that $B \ge 8 n^{n/2} \cdot
    \nu(\Lambda)$ and $B_1 \ge 8 n^2 (n + 1) B$. Assume that $n$ vectors are selected uniformly at
    random from $\Lambda \cap [0, B)^n$ and $n + 1$ vectors uniformly at random from $\Lambda \cap
    [0, B_1)^n$. If the vectors are sampled independently, then the probability that all these
    vectors generate $\Lambda$ is at least 
    \[ 
    	\alpha_n := \biggl( \prod_{i=2}^{n+1} \zeta(i)^{-1} - \tfrac{1}{4} \biggr) \cdot 
	\prod_{k=0}^{n-1} \biggl(1 - n^{k/2} \frac{(4 n^{n/2} + 1)^k}{(4 n^{n/2} - 1)^n}\biggr) \ge 0.092.
    \]
  \end{theorem}
  
  Unfortunately, our current approach requires $m=2n+1$ samples and two windows of different sizes
  to be able to prove that the probability of generating the lattice $\Lambda$ is bounded from below
  by a non-zero constant.  However, based on extensive numerical evidence, we formulate the
  following conjecture, which states that only $m=n+1$ samples and only one window size suffice to
  attain a constant probability of generating the lattice.
  \begin{conjecture}
    For every $n \in \N$, there exists a constant~$0 < c_n < 1$ and a rational function $f_n \in
    \R(x, y)$ satisfying \[ \forall x_0 > 0 \, \forall y_0 \in \bigl(0, x_0^{1/n}\bigr] : \,
    \sup\bigl\{ f_n(x, y) \bigm| 0 < x \le x_0, \, y_0 \le y \le x^{1/n} \bigr\} < \infty \] such
    that the following holds:
    
    Let $\Lambda$ be a lattice in $\R^n$ and let $B > f_n(\det \Lambda, \lambda_1(\Lambda))$. Then
    the probability that $n + 1$ vectors chosen uniformly at random from $\Lambda \cap [0, B)^n$
    generate the lattice~$\Lambda$ is at least $c_n$.
    Moreover, the constant~$c_n$ can be chosen close to $\prod_{k=2}^{n+1} \zeta(k)^{-1}$.
  \end{conjecture}

  \section{Solving the Lattice Generation Problem}
  \label{S:solving}
    
  We break down the lattice generation problem into two subproblems. First, we consider the
  probability that $n$~vectors sampled uniformly at random from $\Lambda$ generate a
  sublattice~$\Lambda_1$ of full rank, i.e. do not lie in a hyperplane. Then, we compute the
  probability that the residue classes of the next $n + 1$~vectors generate the finite abelian
  quotient group~$\Lambda / \Lambda_1$. Finally, we combine these two results.
  
  In the following, we assume that $n > 1$. We discuss a result for the case $n = 1$ in
  Section~\ref{S:Conjecture}.
  
  The idea to prove a lower bound on the probability by considering the above two steps was proposed
  by A.~Schmidt in \cite{arthurDiss}. We present a correct proof of the problem arising in the first
  step, fixing a mistake in Schmidt's proof. Our approach to analyzing the problem arising in the
  second step is entirely different from the approach undertaken by Schmidt. The differences will be
  discussed in Sections~\ref{S:fullsublattice} and \ref{S:final}.
  
  \subsection{Generating a Sublattice of Full Rank}
  \label{S:fullsublattice}
  
  Note that $\lambda_1, \dots, \lambda_n \in \Lambda \cap [0, B)^n$ generate a sublattice of full
  rank if and only if they are linearly independent over $\R$. This is the case if $\lambda_i$ is
  not contained in the $(i - 1)$-dimensional hyperplane spanned by $\lambda_1, \dots,
  \lambda_{i-1}$. Thus to bound the probability that $n$ uniformly random vectors from $\Lambda \cap
  [0, B)^n$ generate a full rank sublattice, we bound the number of lattice elements in the
  intersection as well as the number of lattice elements lying both in the intersection and a
  $k$-dimensional hyperplane, $1 \le k < n$. We find such bounds using Voronoi cells; see also
  Section~1.2 of Chapter~8 in \cite{micciancio-goldwasser}. To state the results, we need to
  introduce some notation, most notably the \emph{covering radius} of a lattice.
    
  Let $\Lambda$ be a lattice in $\R^n$ of full rank. For $\lambda \in \Lambda$, let \[
  V_\Lambda(\lambda) = \{ x \in \R^n \mid \forall \lambda' \in \Lambda \setminus \{ \lambda \} : \|x
  - \lambda\|_2 < \|x - \lambda'\|_2 \} \] be its (open) Voronoi cell. We know that
  $V_\Lambda(\lambda)$ is contained in an open ball of radius $\nu(\Lambda)$ centered around
  $\lambda$, where $\nu(\Lambda)$ is the covering radius of $\Lambda$, and that the volume of
  $V_\Lambda(\lambda)$ is $\det \Lambda$. Moreover, if $\lambda \neq \lambda'$, $V_\Lambda(\lambda)
  \cap V_\Lambda(\lambda') = \emptyset$, and $\bigcup_{\lambda \in \Lambda}
  \overline{V_\Lambda(\lambda)} = \R^n$. Details can be found in
  \cite[Chapter~8]{micciancio-goldwasser}.
  
  Note that $\nu(\Lambda) \le \frac{1}{2} n^{n/2+1} \frac{\det \Lambda}{\lambda_1(\Lambda)^{n-1}}$,
  where $\lambda_1(\Lambda)$ denotes the first successive minimum of $\Lambda$
  \cite{micciancio-goldwasser}, i.e.~the length of a shortest nonzero vector in $\Lambda$.
  
  \begin{lemma}
    \label{L:voronoilemma1}
    If $B > 2 \nu(\Lambda)$. Then \[ \frac{(B - 2\nu(\Lambda))^n}{\det \Lambda} \le \abs{\Lambda
    \cap [0, B)^n} \le \frac{(B + 2\nu(\Lambda))^n}{\det \Lambda}. \]
  \end{lemma}
  
  \begin{lemma}
    \label{L:voronoilemma2}
    Let $B > 0$ and $H$ be a $k$-dimensional hyperplane, $1 \le k < n$. Then \[ \abs{\Lambda \cap H
    \cap [0, B)^n} \le \frac{n^{k/2} (B + 2 \nu(\Lambda))^k (2 \nu(\Lambda))^{n-k}}{\det
    \Lambda}. \]
  \end{lemma}
  
  The proofs are similar to the one of Proposition~8.7 in \cite{micciancio-goldwasser}:
  
  \begin{proof}[Proof of Lemma~\ref{L:voronoilemma1}.]
    \label{P:voronoilemma1}
    If $\lambda \in \Lambda$ satisfies $V_\Lambda(\lambda) \cap [\nu(\Lambda), B - \nu(\Lambda))^n
    \neq \emptyset$, then we must have $\lambda \in [0, B)^n$. Therefore, $ (B - 2 \nu(\Lambda))^n /
    \det \Lambda \le \abs{\Lambda \cap [0, B)^n}$.
    
    If $\lambda \in \Lambda \cap [0, B)^n$, then we must have $V_\Lambda(\lambda) \subseteq
    [-\nu(\Lambda), B + \nu(\Lambda))^n$.  Therefore, $\abs{\Lambda \cap [0, B)^n} \le (B +
    2\nu(\Lambda))^n / \det \Lambda$.
  \end{proof}
  
  \begin{proof}[Proof of Lemma~\ref{L:voronoilemma2}.]
    \label{P:voronoilemma2}
    Let $\lambda \in \Lambda \cap H \cap [0, B)^n$. Then $V_\Lambda(\lambda) \subseteq X :=
    [-\nu(\Lambda), B + \nu(\Lambda))^n \cap (H + \mathrm{B}_{\nu(\Lambda)}(0))$, where
    $\mathrm{B}_{\nu(\Lambda)}(0)$ is a ball of radius~$\nu(\Lambda)$ centered around 0. Therefore,
    $\abs{\Lambda \cap H \cap [0, B)^n} \le \volume(X) / \det \Lambda$, and we have to estimate
    $\volume(X)$.
    
    Clearly, if $\volume_k(Y)$ denotes the $k$-dimensional volume of $Y := H \cap [-\nu(\Lambda), B
    + \nu(\Lambda))^n$, we have that $\volume(X) \le \volume_k(Y) \cdot (2 \nu(\Lambda))^{n -
    k}$. (In fact, we can replace $(2 \nu(\Lambda))^{n - k}$ by the volume of an $(n -
    k)$-dimensional ball of radius~$\nu(\Lambda)$.)
    
    Let $b_1, \dots, b_k$ be an orthonormal basis of $H$. Set $T := \{ (x_1, \dots, x_k) \in \R^k
    \mid \sum_{i=1}^k x_i b_i \in [-\nu(\Lambda), B + \nu(\Lambda))^n \}$; then $\volume(T) =
    \volume_k(Y)$. A point $y \in Y$ corresponds to $(\< y, b_1 \>, \dots, \< y, b_k \>) \in
    T$. Write $b_i = (b_{i1}, \dots, b_{in})$ and $y = (y_1, \dots, y_n) \in [-\nu(\Lambda), B +
    \nu(\Lambda))^n$, set $A_{ij} := B + \nu(\Lambda)$ if $b_{ij} \ge 0$ and $A_{ij} :=
    \nu(\Lambda)$ if $b_{ij} < 0$. Then
    \[ 
    	\sum_{j=1}^n |b_{ij}| (A_{ij} - (B + 2\nu(\Lambda))) \le \< y, b_i \> = 
    	\sum_{j=1}^n y_j b_{ij} \le \sum_{j=1}^n |b_{ij}| A_{ij}, 
    \] 
    implying that $\< y, b_i \>$ ranges over an interval of length $\norm{b_i}_1 (B + 2
    \nu(\Lambda)) \le \sqrt{n} (B + 2 \nu(\Lambda))$. Therefore,
    \[ 
    	\volume(T) \le n^{k/2} (B + 2 \nu(\Lambda))^k. \qedhere
    \]
  \end{proof}
  
  The lemmas allow us to find the following bound on the probability that $n$ random vectors
  generate a sublattice of full rank:
  
  \begin{corollary}
    \label{cor:probGenFullRankSubLattice}
    Assume that $B \ge 8 n^{n/2} \cdot \nu(\Lambda)$. Let
    \begin{align*}
      X :={} & (\Lambda \cap [0, B)^n)^n \\
      \text{and} \quad Y :={} & \{ (y_1, \dots, y_n) \in X \mid \mspan_\R(y_1, \dots, y_n) =
      \R^n \}.
    \end{align*}
    Then $\abs{Y} \ge \tfrac{1}{2} \abs{X}$.
  \end{corollary}
  
  \begin{proof}
    \label{P:probGenFullRankSubLattice}
    Assume that $y_1, \dots, y_k \in X$ are linearly independent, $0 \le k < n$. We bound the
    probability from above that $y_{k+1} \in X$ is not contained in the hyperplane generated by
    $y_1, \dots, y_k$, which is of dimension~$k$. Write $B = j \cdot \nu(\Lambda)$ with $j \ge 8
    n^{n/2}$. By Lemmas~\ref{L:voronoilemma1} and \ref{L:voronoilemma2}, the probability that
    $y_{k+1}$ is in a $k$-dimensional hyperplane is bounded from above by
    \[
      P_k := \frac{n^{k/2} (B + 2 \nu(\Lambda))^k (2 \nu(\Lambda))^{n-k}}{\det \Lambda} \cdot
      \frac{\det \Lambda}{(B - 2 \nu(\Lambda))^n} = n^{k/2} \frac{(j + 2)^k 2^{n - k}}{(j - 2)^n}.
    \]
    The success probability is bounded from below by $\prod_{k=0}^{n-1} (1-P_k)$. Using induction on
    $n$, we can prove that
    \[
    \prod_{k=0}^{n-1} (1-P_k) \ge 1 - \sum_{k=0}^{n-1} P_k. 
    \]
    The sum $\sum_{k=0}^{n-1} P_k$ can be bounded from above as follows:
    \begin{align*}
    \sum_{k=0}^{n-1} P_k 
    & = 
    \frac{2^n}{(j-2)^n} \sum_{k=0}^{n-1} \left( \frac{\sqrt{n}(j+2)}{2}\right)^k \\
    & = 
    \frac{2^n}{(j-2)^n}  \left[\left( \frac{\sqrt{n}(j+2)}{2}\right)^n - 1\right]
    \left[\left( \frac{\sqrt{n}(j+2)}{2}\right) - 1\right]^{-1} \\
    & <
    \frac{2^n}{(j-2)^n}  \left( \frac{\sqrt{n}(j+2)}{2}\right)^n 
    \left[\left( \frac{\sqrt{n}(j+2)}{2}\right) - 1\right]^{-1} \\
    & =
    n^{n/2} \left( 1 + \frac{4}{j-2}\right)^n 
    \left[\left( \frac{\sqrt{n}(j+2)}{2}\right) - 1\right]^{-1}.
    \end{align*}
    Now $\bigl( 1 + \frac{4}{j-2}\bigr)^n \le \exp(\frac{4n}{j-2}) \le \exp(\frac{4}{8 n^{n/2-1} -
    2/n}) \le 2$ for all $n \ge 1$ and $\sqrt{n}(j+2)/2 - 1 \ge j/2$, whence \[ \sum_{k=0}^{n-1} P_k
    < 2 n^{n/2} \cdot \frac{2}{j} \le \frac{4 n^{n/2}}{8 n^{n/2}} = \frac{1}{2}. \qedhere \]
  \end{proof}
  
  Note that our lower bound is far from optimal. If one considers the value~$\sum_{k=0}^{n-1} P_k$
  from the proof and substitutes $j$ by $8 n^{n/2}$, one obtains the lower bound \[
  \prod_{k=0}^{n-1} \biggl( 1 - n^{k/2} \frac{(4 n^{n/2} + 1)^k} {(4 n^{n/2} - 1)^n} \biggr). \] For
  $n = 1$ this is $\frac{2}{3}$, and the product grows to $1$ for $n \to \infty$. For small~$n$, the
  values are:
  \begin{center}
    \begin{tabular}{r|ccccccc}
      Dimension~$n$ & 1 & 2 & 3 & 4 & 5 & 6 & 7 \\\hline
      Lower bound & 0.666 & 0.725 & 0.812 & 0.859 & 0.883 & 0.896 & 0.905
    \end{tabular}
  \end{center}

  \begin{remark}
    The basic idea of the proof of this corollary is similar to the proof of the first part of
    Satz~2.4.23 in \cite{arthurDiss}. Note that the proof in \cite{arthurDiss} is not correct: the
    ratio $\abs{M_{i-1} \cap \calB}/ \abs{M_i \cap \calB}$ considered in the proof can be
    $> \frac{1}{2}$; for example, consider $r = 3$, $M = \Z^3$, $n > 0$ arbitrary (in
    \cite{arthurDiss}, $n \nu(M)$ is what we denote by $b$, i.e., $\calB = [0, n \nu(M))^n$), $x_1 =
    (1, n \nu(M) - 1, 0)$, $x_2 = (0, 1, n \nu(M) - 1)$, $x_3 = (0, 0, 1)$; then $M_1 \cap \calB$
    contains two elements, while $M_2 \cap \calB$ contains three elements. Therefore,
    $\abs{M_1 \cap \calB} / \abs{M_2 \cap \calB} = \frac{2}{3} > \frac{1}{2}$. The problem is
    that $\det M_i$ cannot be bounded linearly in terms of $\nu(M)$ and $\det M_{i-1}$, as it was
    claimed in that proof; in this example, $\det M_1 = \sqrt{1 + (n - 1)^2}$, $\det M_2 = \sqrt{1 +
    (n - 1)^2 + (n - 1)^4}$ and $\nu(M) = 1$. In our proof, we proceed differently by considering
    the ratio $\abs{M_i \cap \calB}/ \abs{M \cap \calB}$ directly, and both our bound on
    the probability and our bound on the minimal size of $\calB$ is in fact better than the
    corresponding bounds given in \cite{arthurDiss}.
  \end{remark}
  
  \subsection{Generating a Finite Abelian Group}
  \label{S:groupgen}
  
  In case $\Lambda_1$ is a sublattice of full rank of $\Lambda$, the quotient group $G = \Lambda /
  \Lambda_1$ is a finite abelian group. Its order equals the index~$[\Lambda : \Lambda_1]$, and by
  the Elementary Divisor Theorem, it can be generated by $n$~elements.
  
  \begin{proposition}\label{prop:genFiniteAbelianGroup}
    Let $G$ be a finite abelian group known to be generated by~$n$ elements. Then the probability
    that $n + 1$ elements drawn uniformly at random from $G$ generate $G$ is at least
    \begin{align*}
      \hat{\zeta} := \prod_{i=2}^\infty \zeta(i)^{-1} \ge 0.434\,,
    \end{align*}
    where $\zeta$ denotes the Riemann zeta function.
  \end{proposition}
  
  For the decimal expansion of $\hat{\zeta}$, see \cite{Sloane}. The probability that a finite group
  is generated by a certain number of random elements has been studied extensively. Formulas for the
  probability for $p$-groups and products of finite groups of coprime orders have been derived by
  V.~Acciaro in \cite[Lemma 4 and Corollary 3]{acciaro} (see also \cite{pomerance-generate}). Our
  result is essentially a corollary of these two results, which we have not found in this form in
  the literature.
    
  \begin{proof}[Proof of Proposition~\ref{prop:genFiniteAbelianGroup}]
    \label{P:genFiniteAbelianGroup}
    For a finite group~$H$, let $\lambda_t(H)$ be the probability that $t$~group elements chosen
    uniformly at random generate~$H$. In \cite{acciaro}, it is shown that if $H$ is a $p$-group with
    minimal number~$d$ of generators, then 
    \begin{align*}
    	\lambda_t(H) = \prod_{i=1}^d (1 - p^{-i}) \cdot
    	\prod_{i=d+1}^t \frac{p^{i-d} - p^{-d}}{p^{i-d} - 1} = \prod_{i=t-d+1}^t (1 - p^{-i})
    \end{align*}
    for $t
    \ge d$ (Lemma~4), and that if $H = H_1 \times H_2$ with $\abs{H_1}, \abs{H_2}$ being coprime,
    then $\lambda_t(H) = \lambda_t(H_1) \lambda_t(H_2)$ (Corollary~3).
    
    Let $p_1, \dots, p_k$ be the distinct prime divisors of $|G|$, and let $G_i$ be the $p_i$-Sylow
    subgroup of $G$. Then $G = G_1 \oplus \dots \oplus G_k$. Now \cite[Corollary~3]{acciaro} yields
    $\lambda_t(G) = \prod_{i=1}^k \lambda_t(G_i)$ since $\abs{G_i}$ is a $p_i$-group, and $p_i \neq
    p_j$ for $i \neq j$. Let $d_i$ be the minimal number of generators for $G_i$; since the minimal
    number for $G$ is~$n$, we must have $d_i \le n$. Thus, by \cite[Lemma~4]{acciaro} \[
    \lambda_{n+1}(G_i) = \prod_{i=n+1-d_i}^{n+1} (1 - p^{-i}) \ge \prod_{i=2}^{n+1} (1 - p^{-i}). \]
    Therefore, the probability that $n$ elements of an arbitrary finite abelian group~$G$ which can
    be generated by $n$ elements generate the group is at least \[ \prod_p \prod_{i=2}^{n+1} (1 -
    p^{-i}) = \prod_{i=2}^{n+1} \prod_p (1 - p^{-i}) = \biggl( \prod_{i=2}^{n+1} \zeta(i)
    \biggr)^{-1} \] using the Euler product representation of the Riemann zeta function.  Now \[
    \prod_{i=2}^{n+1} \zeta(i) \le \prod_{i=2}^{\infty} \zeta(i) = \hat{\zeta}^{-1}. \qedhere \]
  \end{proof}
  
  Observe that our approach only works if we have at least $n + 1$ elements. If we chose just~$n$
  elements randomly, the final product would include $\zeta(1)^{-1} = 0$ and the probability would
  drop down to zero. However, a different approach can result in a non-zero probability for $n$
  elements. This probability will necessarily not be constant anymore, but has to depend on $n$ or
  $\abs{G}$. For example, if $p_1, \dots, p_k$ are distinct primes and $G = \prod_{i=1}^k \F_{p_i}^n
  \cong (\Z/(p_1 \cdots p_k)\Z)^n$, then $G$ can be generated by $n$ elements, but the probability
  that $n$~random elements from $G$ generates~$G$ is exactly $\prod_{i=1}^k \prod_{j=1}^n (1 -
  p_i^{-j})$, which goes to zero for $k \to \infty$ for exactly the above reasons. Hence, any
  non-trivial bound on the probability must take $n$ or $p_1, \dots, p_k$ into account.
  
  This shows that our approach will not work with fewer than $2 n + 1$ elements, if the desired
  bound on the probability should be independent of $n$.
  
  \subsection{The Final Result}
  \label{S:final}
  
  Assume that the first~$n$ sampled vectors from $\Lambda \cap [0, B)^n$ generate a
  sublattice~$\Lambda_1$ of full rank. Then $G = \Lambda/\Lambda_1$ is a finite abelian group which
  can be generated by $n$~elements. Thus if we sample~$n + 1$ elements $\lambda + \Lambda_1$ from
  $G$ in a uniform random manner, we can bound the probability that they generate $G$. In case $G =
  \langle \lambda_{n+1} + \Lambda_1, \dots, \lambda_{2n+1} + \Lambda_1 \rangle$ and $\Lambda_1 =
  \langle \lambda_1, \dots, \lambda_n \rangle$, we have $\Lambda = \langle \lambda_1, \dots,
  \lambda_n, \lambda_{n+1}, \dots, \lambda_{2n+1} \rangle$.
  
  The main problem is that we cannot directly sample uniformly at random from~$G$: if we choose
  $\lambda \in \Lambda \cap [0, B)^n$ uniformly at random, then $\lambda + \Lambda_1$ will in
  general not be uniformly distributed in $G = \Lambda/\Lambda_1$. By enlarging the window~$[0,
  B)^n$ to $[0, B_1)^n$ with $B_1 > B$ large enough, we ensure that the residue classes of the
  samples $\lambda \in \Lambda \cap [0, B_1)^n$ are essentially distributed uniformly at random in
  $G$. More precisely, we can show that the \emph{statistical distance} between the distribution and
  the perfectly uniform distribution is small enough. This is established by the following result:
  
  \begin{lemma}\label{lem:samplingQuotient}
    Let $\Lambda_1$ be an arbitrary full-rank sublattice of $\Lambda$. Assume that $B_1 > 2
    \nu(\Lambda_1)$ and we can sample uniformly at random from $\Lambda \cap [0, B_1)^n$. Denote the
    sample by $\lambda$.
    Then, the total variation distance between the
    uniform distribution over $\Lambda/\Lambda_1$ and the distribution of $\lambda + \Lambda_1$,
    where $\lambda \in \Lambda \cap [0, B_1)^n$ is uniformly distributed, is at most
    \begin{align*}
      1 - \frac{(B_1 - 2 \nu(\Lambda_1))^n}{(B_1 + 2 \nu(\Lambda))^n}\,.
    \end{align*}
  \end{lemma}
  
  \begin{proof}
    \label{P:samplingQuotient}
    First note that $V_{\Lambda_1}(\lambda_1) = \lambda_1 + V_{\Lambda_1}(0)$ and
    $\overline{V_{\Lambda_1}(\lambda_1)} = \lambda_1 + \overline{V_{\Lambda_1}(0)}$. Now, as
    $\bigcup_{\lambda_1 \in \Lambda_1} (\lambda_1 + \overline{V_{\Lambda_1}(0)}) = \R^n$ and two
    translates of $V_{\Lambda_1}(0)$ by different elements of $\Lambda_1$ do not intersect, there
    exists a set $V$ with $V_{\Lambda_1}(0) \subseteq V \subseteq \overline{V_{\Lambda_1}(0)}$
    satisfying \[ \bigcup_{\lambda_1 \in \Lambda_1} (\lambda_1 + V) = \R^n \quad \text{and} \quad
    \forall \lambda_1 \in \Lambda_1 \setminus \{ 0 \} : (\lambda_1 + V) \cap V = \emptyset. \] Note
    that $\volume(V) = \volume(V_{\Lambda_1}(0)) = \det \Lambda_1$.
    
    We first assume that the window has the form $[0,B_1]^n$ instead of the form $[0,B_1)^n$ and
    later argue that the bounds derived also apply to the actual window $[0,B_1)^n$.  We need the
    following three facts:
    \begin{itemize}
    \item There are exactly $m=\det \Lambda_1 / \det \Lambda$ points of $\Lambda$ in each translate of $V$, i.e., 
    \[
    	| (\lambda_1 + V) \cap \Lambda | = m \mbox{ for all $\lambda_1\in\Lambda_1$.}
    \]
    This can be shown by using asymptotic arguments similarly to those used in the proof that each
    translate of the elementary parallelepipeds of $\Lambda_1$ contains exactly $m$ elements of
    $\Lambda$ (see e.g.\ \cite{barvinok-notes}).
    
    \item There are at least 
    \[ 
    	\ell = \frac{(B_1 - 2\nu(\Lambda_1))^n}{\det \Lambda_1} 
     \] 
     translates of $V$ that are entirely contained inside the window $[0, B_1]^n$ since 
     $V \subseteq \overline{B_{\nu(\Lambda_1)}(0)}$.
    
    \item There are at most 
    \[ 
    	u = \frac{(B_1 + 2\nu(\Lambda))^n}{\det \Lambda} 
    \] 
    points of $\Lambda$ inside $[0,B_1]^n$.
    \end{itemize}
    Let $\Omega=\Lambda\cap [0,B_1]^n$.  We call $\lambda\in\Omega$ \emph{good} if there exists
    $\lambda_1\in\Lambda_1$ such that
    \[
    	\lambda \in \lambda_1 + V \subseteq [0,B_1]^n.
    \]
    In words, $\lambda\in\Omega$ is good if it belongs to a translate of $V$ that is entirely 
    inside the window $[0,B_1]^n$.  Let $\Omega_{\rm{good}}$ denote the
    set of good points.  Using the first two facts above, we deduce that $|\Omega_{\rm{good}}|\ge m \ell$.
    
    Let $\calP$ denote the uniform distribution on $\Omega$ and $\tilde{\calP}$ the uniform
    distribution on the set of good points.  We view $\tilde{\calP}$ as a probability distribution
    on $\Omega$ by assigning the probability $0$ to any point that is not good.  Then, the total
    variation distance between $\calP$ and $\tilde{\calP}$ is bounded from above by
    \begin{align*}
    	\frac{1}{2} \sum_{\lambda\in\Omega} | \calP(\lambda) - \tilde{\calP}(\lambda) | 
	&=
	\frac{1}{2} |\Omega_{{\rm good}}| \left( \frac{1}{|\Omega_{\rm{good}}|} - \frac{1}{|\Omega|} \right) + 
	\frac{1}{2} \left( |\Omega| - |\Omega_{{\rm good}}| \right) \frac{1}{|\Omega|} \\
	&= 1 - \frac{|\Omega_{\rm{good}}|}{|\Omega|} 
	\le 1 - \frac{m\ell}{u} 
	\le 1 - \frac{\big(B_1-2\nu(\Lambda_1)\big)^n}{\big(B_1+2\nu(\Lambda)\big)^n}.
    \end{align*}
    
    Let $\kappa : \Lambda \rightarrow \Lambda/\Lambda_1$ denote the canonical projection map.  Let
    $\cal{Q}$ and $\tilde{\calQ}$ be the probability distribution on the cosets $\Lambda/\Lambda_1$
    induced by the following two-step process: (1) sample $\lambda$ according to $\calP$ and
    $\tilde{\calP}$, respectively, and (2) apply $\kappa$ to the obtained sample $\lambda$.  Observe
    that $\tilde{\calQ}$ is the uniform distribution on $\Lambda/\Lambda_1$.  Unfortunately, we
    cannot sample according to $\tilde{\calQ}$ but only according to $\calQ$.  However, the total
    variation distance between $\calQ$ and $\tilde{\calQ}$ must be less or equal to the one between
    $\calP$ and $\tilde{\calP}$ since the total variation distance satisfies a so-called data
    processing inequality.
        
    Note that so far, we have considered $[0, B_1]^n$ instead of $[0, B_1)^n$. As $\Lambda$ is
    discrete, there exists some $2 \nu(\Lambda_1) < B_1' < B_1$ with $[0, B_1']^n \cap \Lambda = [0,
    B_1)^n$. Applying the result above to $[0, B_1']^n$ and then using that 
    \begin{align*}
    	x \mapsto 1 - \frac{(x
    	- 2 \nu(\Lambda_1))^n}{(x + 2 \nu(\Lambda))^n}
    \end{align*} 
    is increasing yields the stated claim for $[0,B_1)^n$.
  \end{proof}
  
  Combining the lemma and Proposition~\ref{prop:genFiniteAbelianGroup} and using the additivity of
  the total variation distance under composition provided that the components are independent, we
  obtain the following result:
  
  \begin{corollary}\label{corr:generateWholeLattice}
    Assume that $B \ge 8 n^{n/2} \cdot \nu(\Lambda)$ and $B_1 \ge 8 n^2 (n + 1) B$. Let $Y$ be as in
    Corollary~\ref{cor:probGenFullRankSubLattice} and $(y_1,\ldots,y_n)\in Y$. Let
    \begin{eqnarray*}
      X_1 & := & \big(\Lambda \cap [0, B_1)^n \big)^{n+1} \\
      Z   & =  & \{(z_1,\ldots,z_{n+1})\in X_1^{n+1} \mid \mspan_\Z
      \{y_1,\ldots,y_n,z_1,\ldots,z_{n+1} \} = \Lambda \}.
    \end{eqnarray*}
    Then $\abs{Z} \ge \bigl(\hat{\zeta}-\tfrac{1}{4} \bigr) \abs{X_1} \ge 0.184 \abs{X_1}$.
  \end{corollary}  
  
  \begin{proof}
    \label{P:generateWholeLattice}
    Let $\Lambda_1$ be the full-rank sublattice generated by $y_1,\ldots,y_n$.  We have the
    following simple bound on the covering radius
    \begin{align*}
      \nu(\Lambda_1) \le \frac{\sqrt{n}}{2} \lambda_n(\Lambda_1) \le \frac{\sqrt{n}}{2}
      \max_{i=1,\ldots,n} \| y_i \|_2 \le \frac{\sqrt{n}}{2} \sqrt{n} B = \frac{n B}{2}
    \end{align*}
    since the $y_i$ are linearly independent and every vector in $[0, B)^n$ is shorter than
    $\sqrt{n} B$. Moreover, $\nu(\Lambda_1) \ge \nu(\Lambda)$.
    
    Let $z_i$ be uniformly distributed in $\Lambda \cap [0, B_1)^n$.  Then,
    Lemma~\ref{lem:samplingQuotient} implies that $z_i + \Lambda_1$ (for $i=n+1,\ldots,2n+1$) are
    distributed almost uniformly at random from $\Lambda/\Lambda_1$.  The total variation distance
    from the uniform distribution is bounded from above as follows:
    \begin{align*}
      & 1 - \frac{(B_1 - 2\nu(\Lambda_1))^n}{(B_1 + 2\nu(\Lambda))^n} \le 1 - \frac{(B_1 -
        2\nu(\Lambda_1))^n}{(B_1 + 2\nu(\Lambda_1))^n} = 1 - \left( 1 - \frac{4\nu(\Lambda_1)}{B_1 +
        2\nu(\Lambda_1)} \right)^n \\
      {}\le{} & 1 - \left( 1 - n \, \frac{4\nu(\Lambda_1)}{B_1 + 2\nu(\Lambda_1)} \right) \le
      \frac{4n \nu(\Lambda_1)}{B_1} \le \frac{2 n^2 B}{B_1} \le \frac{1}{4(n+1)}\,.
    \end{align*}
    Consider now the uniform probability distribution on the $(n+1)$-fold direct product of
    $\Lambda/\Lambda_1$ and the probability distribution that arises from sampling almost uniformly
    at random on each of the components as above.  Then the total variation between these two
    distributions is bounded from above by $(n+1) \cdot \frac{1}{4(n+1)} = \tfrac{1}{4}$.  This is
    because the total variation distance is subadditive under composition provided that the
    components are independent (see e.g.\ \cite[Subsection 1.3 ``Statistical distance'' in
    Chapter~8]{micciancio-goldwasser} for more information on the total variation distance).
    
    Clearly, the abelian group $\Lambda/\Lambda_1$ can be generated with only $n$ generators.
    Hence, Proposition~\ref{prop:genFiniteAbelianGroup} implies that $n+1$ samples (provided that
    they are distributed uniformly at random over the group) form a generating set with probability
    greater or equal to $\hat{\zeta}$.  Due to the deviation from the uniform distribution on the
    $(n+1)$-fold direct product of $\Lambda/\Lambda_1$ this probability may decrease. However it is
    at least $\hat{\zeta}-1/4$ since the total variation distance is at most $1/4$.  The claim
    follows now by translating the lower bound on the probability to a lower bound on the fraction
    of elements with the desired property.
  \end{proof}
  
  Combining this corollary with Corollary~\ref{cor:probGenFullRankSubLattice}, we obtain a proof of
  Theorem~\ref{thm:main}.  This theorem is similar to Satz~2.4.23 in \cite{arthurDiss}.  We
  emphasize that our bound on the success probability is constant, whereas the bound presented in
  Satz~2.4.23 decreases exponentially fast with the dimension $n$.  The first part of proof of
  Satz~2.4.23 (concerning the generation of a full-rank sublattice) is unfortunately not correct,
  but can be corrected as we have shown in our proof of
  Corollary~\ref{cor:probGenFullRankSubLattice}.  The idea behind the second part is completely
  different from our proof and cannot be used to prove a constant success probability.  Perhaps it
  could be used to prove that only $2 n$ random elements (as opposed to $2 n + 1$ elements) are
  needed to guarantee a non-zero success probability.
  
  Note that for a fixed dimension~$n$, one obtains bounds larger than 0.092.
  For $n = 2$, $3$, $4$ and $5$, $\alpha_n$ is larger than $0.238$,
  $0.185$, $0.176$, $0.172$ and $0.170$, respectively.
  
  \section{Conjecture}
  \label{S:Conjecture}
  
  Let $b_1, \dots, b_n$ be any basis of the lattice~$\Lambda$. Consider the natural isomorphism
  $\Phi : \R^n \to \R^n$ mapping the $i$-th standard unit vector~$e_i$ to $b_i$. Then $\Phi(\Z^n) =
  \Lambda$. Let \[ X := \Phi^{-1}([0, B)^n) = \biggl\{ (a_1, \dots, a_n) \in \R^n \biggm|
  \sum_{i=1}^n a_i b_i \in [0, B)^n \biggr\}; \] this is a parallelepiped in $\R^n$ of volume
  $\frac{B^n}{\det \Lambda}$ having 0 as a vertex. If we assume that the basis $b_1, \dots, b_n$ is
  reduced, then this parallelepiped is not too skewed.
  
  Now let $v_1, \dots, v_m \in \Lambda$ be vectors, $m \ge n$, and consider $\hat{v}_i :=
  \Phi^{-1}(v_i) \in \Z^n$ for $i = 1, \dots, m$. We have that $\langle v_1, \dots, v_m \rangle =
  \Lambda$ if and only if $\langle \hat{v}_1, \dots, \hat{v}_m \rangle = \Z^n$, and this is the case
  if and only if the matrix $(\hat{v}_1, \dots, \hat{v}_m) \in \Z^{n \times m}$ is
  \emph{unimodular}.
  
  Therefore, the probability that $m \ge n$ vectors selected uniformly at random in $\Lambda \cap
  [0, B)^n$ generate $\Lambda$ equals the probability that an $n \times m$ integer matrix whose
  columns are chosen uniformly at random in $X$ is unimodular.
  
  As indicated in the introduction, this problem was studied for special forms of $X$
  asymptotically. For $X = [-B, B]^n$, $B \to \infty$, it was studied in G.~Maze, J.~Rosenthal and
  U.~Wagner showed in \cite{naturaldensityunimodular}, and for $X = v + [-B, B]^n$, $B \to \infty$
  while $v$ is bounded polynomially in terms of $B$, in S.~Elizalde and K.~Woods
  \cite{elizalde-woods}. In both cases, it was shown that the limit of the probability for $B \to
  \infty$ is $\prod_{j=m-n+1}^m \zeta(j)^{-1}$ -- which for $m = n + 1$, not very surprisingly,
  equals the probability given in Section~\ref{S:groupgen}. This can be bounded from below by
  $\hat{\zeta} > 0.434$ as soon as $m > n$. This implies that for a certain $\hat{B} > 0$, we have
  that the probability is at least $0.434$ for all $B > \hat{B}$.
  
  While it seems probable that the proof of \cite{elizalde-woods} can yield effective non-trivial
  bounds for any such $\hat{B}$. However, it is unclear whether this would help for the general
  case, as the proof only considers the special case~$X = v + [-B, B]^n$, while we have to consider
  essentially arbitrary parallelepipeds with $0$ as a vertex.
  
  We have run computer experiments to study the probability for arbitrary parallelepipeds. We
  restricted to the case~$m = n + 1$. For the experiments, we generated a random parallelepiped by
  choosing $n$~vectors from $[-C, C]^n$ and considering the parallelepiped spanned by them. We
  generated 1000 such parallelepipeds, and for every parallelepiped we generated 10\,000~integer
  matrices with columns taken uniformly at random from the parallelepiped. Every matrix was tested
  whether it is unimodular. We used three different bounds for $C$, namely $C = 10^4$, $C = 10^9$
  and $C = 10^{18}$. For every combination of $n \times m = n \times (n + 1)$ and $C$, we computed
  both the average probability that an $n \times m$ integer matrix taken from a parallelepiped is
  unimodular, and the minimal probability (over all parallelepipeds for given $n \times m$ and
  $C$). The results are shown in Tables~\ref{T:average} (average probabilities) and \ref{T:minimal}
  (minimal probabilities) on page~\pageref{T:average}. They also include the ``ideal'' probabilities
  $\prod_{j=2}^{n+1} \zeta(j)^{-1}$ predicted for the special parallelepiped with $B \to \infty$ in
  \cite{naturaldensityunimodular}.
  
  \begin{table}[p]
    \begin{center}
    \begin{tabular}{|r||r|r|r||r|}\hline
      $n$ & $C = 10^4$ & $C = 10^9$ & $C = 10^{18}$ & ideal probability \\\hline\hline
       1 & 60.7273\% & 60.8094\% & 60.8103\% & 60.7927\% \\\hline
       2 & 50.5849\% & 50.5899\% & 50.5649\% & 50.5739\% \\\hline
       3 & 46.7040\% & 46.7257\% & 46.7367\% & 46.7272\% \\\hline
       4 & 45.0382\% & 45.0252\% & 45.0080\% & 45.0631\% \\\hline
       5 & 44.2531\% & 44.2315\% & 44.2052\% & 44.2949\% \\\hline
       6 & 43.8661\% & 43.8894\% & 43.8740\% & 43.9281\% \\\hline
       7 & 43.6945\% & 43.6773\% & 43.7059\% & 43.7497\% \\\hline
       8 & 43.6003\% & 43.6162\% & 43.6049\% & 43.6620\% \\\hline
       9 & 43.5529\% & 43.5662\% & 43.5447\% & 43.6187\% \\\hline
      10 & 43.5369\% & 43.5343\% & 43.5332\% & 43.5971\% \\\hline
      11 & 43.5124\% & 43.5463\% & 43.5556\% & 43.5864\% \\\hline
      12 & 43.5314\% & 43.5488\% & 43.5218\% & 43.5810\% \\\hline
      13 & 43.5329\% & 43.5314\% & 43.5224\% & 43.5784\% \\\hline
      14 & 43.5217\% & 43.5322\% & 43.5679\% & 43.5770\% \\\hline
      15 & 43.5113\% & 43.5273\% & 43.4947\% & 43.5764\% \\\hline
    \end{tabular}
    \end{center}
    \caption{Average empirical probability that a random $n \times (n + 1)$ integer matrix from a
    random parallelepiped inside $[-C, C]^n$ is unimodular.}
    \label{T:average}
  \end{table}
  
  \begin{table}[p]
    \begin{center}
    \begin{tabular}{|r||r|r|r||r|}\hline
      $n$ & $C = 10^4$ & $C = 10^9$ & $C = 10^{18}$ & ideal probability \\\hline\hline
       1 & 58.98\% & 59.17\% & 59.31\% & 60.7927\% \\\hline
       2 & 49.03\% & 48.91\% & 49.17\% & 50.5739\% \\\hline
       3 & 45.16\% & 44.96\% & 45.34\% & 46.7272\% \\\hline
       4 & 43.09\% & 43.31\% & 43.60\% & 45.0631\% \\\hline
       5 & 42.39\% & 42.61\% & 42.61\% & 44.2949\% \\\hline
       6 & 42.27\% & 42.06\% & 42.06\% & 43.9281\% \\\hline
       7 & 42.24\% & 42.37\% & 41.72\% & 43.7497\% \\\hline
       8 & 41.99\% & 42.17\% & 41.83\% & 43.6620\% \\\hline
       9 & 42.18\% & 42.14\% & 41.78\% & 43.6187\% \\\hline
      10 & 42.14\% & 42.02\% & 42.14\% & 43.5971\% \\\hline
      11 & 41.94\% & 41.97\% & 42.09\% & 43.5864\% \\\hline
      12 & 41.86\% & 41.81\% & 42.09\% & 43.5810\% \\\hline
      13 & 41.98\% & 42.12\% & 42.05\% & 43.5784\% \\\hline
      14 & 41.65\% & 42.10\% & 42.06\% & 43.5770\% \\\hline
      15 & 41.99\% & 42.00\% & 42.13\% & 43.5764\% \\\hline
    \end{tabular}
    \end{center}
    \caption{Minimal empirical probability that a random $n \times (n + 1)$ integer matrix from a
    random parallelepiped inside $[-C, C]^n$ is unimodular.}
    \label{T:minimal}
  \end{table}
  
  As one can clearly see, the average values are very close to the ideal ones. But also the minimal
  probabilities observed in the experiments were always close to the ideal values. In fact, the
  difference between minimal and maximal probabilities never exceeded 3.66\%. If one compares these
  probabilities to the ones given at the end of Section~\ref{S:final}, one sees that the
  probabilities obtained there are far too low.
  
  Our conjecture is based on the evidence sketched above.  
  The conditions on $f$ ensure that given a family of lattices where we have an upper bound on $\det
  \Lambda$ and a lower bound on $\lambda_1(\Lambda)$, we can find a lower bound on $B$ such that the
  result holds for all lattices of this family. This is for example the case for unit lattices of
  number fields. There, one has a lower bound on $\lambda_1(\Lambda)$ depending only on the degree
  of the number field \cite{remak-regulator}, and an upper bound on $\det \Lambda$ in terms of the
  degree and discriminant of the number field \cite{sands-generalization}.
  
  The only case in which we know how to prove the conjecture is $n = 1$. In that case, we have
  $\Lambda = v \Z$ for some real number $v > 0$. Given two elements $a v, b v \in \Lambda \cap [0,
  B)$, we have that $\langle a v, b v \rangle = v \Z$ if and only if $a$ and $b$ are
  coprime. Therefore, we are interested in the probability that two random integers in $\bigl[0,
  \frac{B}{\det \Lambda}\bigr)$ are coprime. For $\frac{B}{\det \Lambda} \to \infty$, it is
  well-known that this probability goes to $\zeta(2)^{-1} = \frac{6}{\pi^2} \approx 0.607927$. One
  can easily make this more precise, for example by using the computations from \cite{lehmer-aects}
  and additional computer computations for $n \le 1000$:
  
  \begin{proposition}
    \label{prop:none}
    Let $n \ge 1$ be a natural number and \[ p_n = \frac{|\{ (x, y) \in \N^2 \mid 0 \le x, y \le n,
    \; \gcd(x, y) = 1 \}|}{(n + 1)^2}. \] Then \[ p_n \ge \frac{13}{22} > 0.5909 \] with equality in
    the first inequality if and only if $n = 10$.
  \end{proposition}
  
  \begin{proof}[Proof of Proposition~\ref{prop:none}.]
    \label{P:none}
    For $n \ge 1$, let \[ A(n) := \abs{\{ (x, y) \in \N^2 \mid 0 \le x, y \le n, \; \gcd(x, y) = 1
    \}}. \] Clearly, $p_n = \frac{A(n)}{(n + 1)^2}$ and $A(n) = 2 \sum_{k=1}^n \phi(k) + 1$,
    where \[ \phi(k) = \abs{\{ x \in \N \mid 0 \le x < k, \; \gcd(x, k) = 1 \}} \] is Euler's
    totient function. Now in \cite[Theorem~IV and proof]{lehmer-aects}, it is proven that \[
    \sum_{k=1}^n \phi(k) = \frac{n^2}{2} \cdot \frac{1}{\zeta(2)} + \Delta(n), \quad \text{where }
    \abs{\Delta(n)} \le n \sum_{k=1}^n \frac{1}{k} + \frac{n^2}{2} \cdot \frac{1}{n} \] and $\zeta$
    is the Riemann $\zeta$ function. Now $\sum_{k=1}^n \frac{1}{k} \le 1 + \int_1^n \frac{1}{x} \;
    dx = 1 + \log n$, whence \[ \abs{\Delta(n)} \le n (1 + \log n) + \tfrac{1}{2} n = \tfrac{3}{2} n
    + n \log n. \] This together with $\zeta(2) = \frac{\pi^2}{6}$ shows that
    \begin{align*}
      p_n ={} & \frac{1 + 2 \sum_{k=1}^n \phi(k)}{(n + 1)^2} \ge \frac{1 + 2 \bigl( \frac{3}{\pi^2}
        n^2 - \tfrac{3}{2} n - n \log n \bigr)}{(n + 1)^2} \\ {}={} & \frac{6}{\pi^2} \cdot
      \frac{n^2}{(n + 1)^2} - \frac{n \log n}{(n + 1)^2} - \frac{3 n}{2 (n + 1)^2} + \frac{1}{(n +
        1)^2}.
    \end{align*}
    Using a computer program, one quickly verifies that $p_n \ge \frac{13}{22}$ for all $n \in \Z
    \cap [1, 1000]$, with equality if and only if $n = 10$. For $n > n_0 := 1000$, the above
    inequality yields \[ p_n > \frac{6}{\pi^2} \cdot \frac{n_0^2}{(n_0 + 1)^2} - \frac{n_0 \log
    n_0}{(n_0 + 1)^2} - \frac{3 n_0}{2 (n_0 + 1)^2} > \tfrac{13}{22}. \qedhere \]
  \end{proof}
   
  Therefore, the conjecture is true for $n = 1$ with $c_1 = \frac{13}{22}$ and $f_1(x, y) = x$.
  
  Finally, note that in case $n = m$, the result in \cite{naturaldensityunimodular} shows that one
  expects that the only lower bound one can give is 0. We have run a few experiments here as well,
  and already for $C = 10^4$, not a single unimodular matrix was found during the experiments.
  
  \section{Relevance of Lattice Generation to Quantum Algorithms and Quantum Cryptanalysis}
  The \emph{Discrete Logarithm Problem} (DLP) is a mathematical primitive on which many public-key
  cryptosystems are based. Examples of groups, in which the DLP is considered to be computationally
  hard, include the multiplicative group of $\F_q$ \cite{HOAC}, the group of $\F_q$-rational points
  of an elliptic curve \cite{hehcc}, the divisor class and ideal class groups of an algebraic curve, 
  or the infrastructure of an algebraic number field
  \cite{buchmann-numbertheoreticalgsandcrypto,scheidler-buchmann-williams-qrkeyexchange}. 
  For the cryptographically relevant instances, the best know classical algorithms have a subexponential running time. 
  In contrast, there are efficient quantum algorithms that solve these
  DLPs in polynomial time
  \cite{shor,cheung-mosca,hallgrenPell,pradeep-pawel,hallgrenUnitgroup,schmidt-vollmer,arthurDiss}.
  
  The statement on the running time of the quantum algorithm for solving the DLP in the
  infrastructure of a number field needs to be made more precise.  It scales polynomially in the
  logarithm of the discriminant~$\Delta_{K/\Q}$ of the number field~$K$, but exponentially in a
  polynomial expression~$q([K : \Q])$ of its degree~$[K : \Q]$.  While the exponential scaling seems
  to be unavoidable for fundamental reasons -- one has to compute shortest lattice basis vectors in
  dimension $[K : \Q]$ to be able to perform basic arithmetic operations in the infrastructure -- it
  is important to reduce the magnitude of the polynomial~$q$.  We now explain why our theorem and
  conjecture on lattice generation can be used to achieve such reduction.
  
  The quantum algorithm solves the DLP problem by reducing it to the problem of finding a basis of a
  certain full-rank lattice $L\subset\R^n$ where $n=d+1$.  We explain this reduction in more detail
  at the end of this section.  To find a basis of $L$, the quantum algorithm has a mechanism which,
  with a certain probability~$p_1 > 0$, outputs an essentially uniformly distributed vector~$\lambda
  \in \Lambda \cap [0, B)^{n+1}$, where $\Lambda=L^\ast$ is the dual lattice of $L$ and $B > 0$ is
  sufficiently large.  With probability $1-p_1$ it outputs a vector that is not an element of
  $\Lambda$. Unfortunately, this unfavorable case cannot be recognized efficiently.  If one has
  $\lambda_1, \dots, \lambda_m$ with $\Lambda = \langle \lambda_1, \dots, \lambda_m \rangle_\Z$, one
  can compute a basis of $\Lambda$ from these vectors and then use linear algebra (matrix inversion)
  to retrieve a basis of $L=\Lambda^\ast$ itself.
  
  To compute the success probability, one has to consider the probability that the
  $m$ sampled vectors are actually in $\Lambda$, and the probability that the $m$ random vectors
  from $\Lambda \cap [0, B)^{n+1}$ generate $\Lambda$. If the latter probability is $p_2$,
  then the overall success probability is $\approx p_1^m p_2$, and one expects that one has to run
  the algorithm $\approx (p_1^m p_2)^{-1}$ times before it outputs a basis of $\Lambda$ and
  thus of $L$ itself.
  The main problem is that for $n > 1$, the lower bound one can prove for $p_1$ is quite small. 
  In fact, it seems unavoidable that $p_1$ is bounded away from $1$ by a nonzero constant.  Therefore,
  it becomes evident why it is so important to minimize $m$ without decreasing $p_2$ too much.
  
  Theorem~\ref{thm:main} shows that the quantum algorithm can recover a basis of $L$ with constant
  probability $p_2$ conditioned on the event that it has obtained $m=2n + 1$ samples of $\Lambda$,
  which occurs with probability $p_1^m=p_1^{2n+1}$.  For this, we use two different window sizes:
  the first $n$ vectors are sampled from a smaller window $[0, B)^{n+1}$, and the latter $n +
  1$~vectors from a larger window $[0, B_1)^{n+1}$ with $B_1 > B$. To the best of our knowledge,
  this is the first explicit result that leads to a rigorous bound on the running time of the
  quantum algorithm.
    
  Conjecture implies that that the quantum algorithm can recover a basis of $L$ with constant probability $p_2$
  conditioned on the event that it has obtained only $m_c=n+1$ samples of $\Lambda$.  This event occurs with
  probability $p_1^{m_c}=p_1^{n+1}$, which is greater by the exponential factor $1/p_1^n$.  Moreover, the 
  quantum algorithm becomes simpler because it suffices to sample vectors
  from one window. 
  
  For the sake of completeness, we now describe the reduction in more detail.  The infrastructure of
  a number field is isomorphic to a torus~$T = \R^d/M$, where $M$ is a full-rank lattice in $\R^d$
  and the coefficients of all non-trivial vectors of $M$ are transcendental numbers
  \cite{ff-tioagfoaur}.  This forces one to work with approximations, which is ultimately
  responsible for the poor performance when the dimension $d$ increases.
    
  Assume that we are given two elements~$x, y \in T$ and we want to find the ``discrete
  logarithm''~$\ell \in \Z$ with $\ell x = y$, assuming that such number $\ell$ exists. For this, it suffices to
  know $v, w \in \R^d$ with $v + M = x$ and $w + M = y$ together with a basis of
  $M$: then one can use linear algebra to recover $\ell$ (and decide whether it exists). In the
  infrastructure of a number field, one has a representation of $T$ which allows us to compute the
  projection $\R^d \to T$ easily, but recovering a preimage of a random~$x \in T$ is hard.
  
  Finding a preimage can be reformulated as a lattice problem: consider the map \[ \phi : \Z \times
  \R^d \to T, \quad (z, v) \mapsto z x + v; \] this is a group homomorphism, and the
  kernel~$L := \ker \phi$ is a lattice in $\R^{d+1}$ of full rank. The kernel contains
  elements of the form $(-1, v)$ with $v \in \R^d$; any such element satisfies $v + M = x$. If
  we have a basis of $L$, we can again use linear algebra to recover such an element. Thus,
  the task is to find a basis of a lattice $L \subseteq \R^{d+1}$ of full rank.

  Finally, we want to note that using different distributions on the lattice vectors can lead to
  much better results which are simpler to obtain. For example, when using the discrete Gaussian
  distribution on the lattice points, a result similar to ours in Theorem~\ref{thm:main} follows
  from works by D.~Micciancio, O.~Regev \cite{micciancio-regev}, C.~Gentry, C.~Peikert and
  V.~Vaikuntanathan \cite{gentry-peikert-vaikuntanathan}. Unfortunately, it is not known how to
  sample from this distribution on a quantum computer, even in case a basis of the lattice is
  given. For our problem, where we want to determine a basis, we are only given an indirect
  description of the lattice. Therefore, these results cannot be used for solving the DLP without
  new ideas.

  \paragraph{Acknowledgments}
  P.W. gratefully acknowledges the support from the NSF grant CCF-0726771 and the NSF CAREER Award
  CCF-0746600. P.W. would also like to thank Joachim Rosenthal and his group members for their
  hospitality during his visit at the Institute of Mathematics, University of
  Zurich. F.F. gratefully acknowledges partial support form the SNF grant No.~132256. Both authors
  would like to thank A.~Schmidt for pointing out an error in an earlier preprint, both referees for
  their valuable comments, and S.~Elizalde and K.~Woods for pointing us to their paper
  \cite{elizalde-woods}.
  
\newcommand{\etalchar}[1]{$^{#1}$}

\end{document}